\documentclass[a4paper,reqno,10pt]{amsart}
\usepackage{amsfonts,amssymb,amsthm,amsbsy,latexsym,amscd,amsmath,dsfont,euscript,enumerate,manfnt, marvosym,verbatim,calc,mathrsfs,marvosym,textcomp ,color,xcolor,tikz,tikz-cd,setspace,stmaryrd,upgreek,bm}

\usepackage{graphicx}
\usepackage{circuitikz}

\definecolor{bulgarianrose}{rgb}{0.28, 0.02, 0.03}
\usepackage[linkcolor=bulgarianrose,citecolor=bulgarianrose,colorlinks=true,hypertexnames=false]{hyperref}

\newtheorem{theorem}{Theorem}[section]
\newtheorem{corollary}[theorem]{Corollary}
\newtheorem{lemma}[theorem]{Lemma}

\theoremstyle{definition}

\newtheorem*{remark}{Remark}

\newtheorem*{conjecture}{Conjecture}

\newtheorem*{acknowledgement}{Acknowledgement}

\newcommand{\ex}{\ensuremath{{\textnormal{ex}}}}

\makeatletter
\def\imod#1{\allowbreak\mkern10mu({\operator@font mod}\,\,#1)}
\def\@textbottom{\vskip\z@\@plus 18pt}
\let\@texttop\relax
\makeatother

\title[Ramsey numbers involving odd cycles]{A study of two Ramsey numbers\\ involving odd cycles}
\author{Sayan Gupta}

\address{\newline 
\newline (a) School of Mathematical Sciences\\ \newline National Institute of Science Education and Research (NISER) Bhubaneswar\\\newline Jatni, Khurda- $752050$, Odisha, India.
\newline (b) Homi Bhabha National Institute (HBNI)\\ \newline Training School Complex, Anushakti Nagar, Mumbai- $400094$, India.
\newline \textnormal{\textestimated-Mails}: {\tt sayan.gupta\MVAt niser.ac.in, sayangupta4u\MVAt gmail.com}}


\begin{document}

\thispagestyle{empty}
\vspace{10cm}

\begin{abstract}
For any two graphs $G$ and $H$, the \emph{Ramsey number} $R(G,H)$ is the minimum integer $n$ such that any graph on $n$ vertices either contains a copy of $G$ or its complement contains a copy of $H$ as a subgraph. The \emph{book graph} of order $(n+2)$, denoted by $B_{n}$, is the graph with $n$ distinct copies of triangles sharing a common edge called the `base'. A cycle of order $m$ is denoted by $C_{m}$. A lot of studies have been done in recent years on the Ramsey number $R(B_{n}, C_{m})$. However, the exact value remains unknown for several $n$ and $m$. In 2021, Lin and Peng obtained the value of $R(B_{n}, C_{m})$ under certain conditions on $n$ and $m$. In this paper, they remarked that the value is still unknown for the range $n\in [\frac{9m}{8}-125, 4m-14]$. In a recent paper, Hu et al. determined the value of the book-cycle Ramsey number within the range $n\in [ \frac{3m-5}{2}-125, 4m]$ where $m$ is odd and $n$ is sufficiently large. In this article, we extend the investigation to smaller values of $n$. We have obtained a bound of $R(B_{n}, C_{m})$ if $n\in [2m-3, 4m-14]$ and $m\geq 7$ is odd. This is a progress on the earlier result. A connected graph $G$ is said to be \emph{$H$-good} if the formula,
\begin{equation*}
R(G,H)= (|G|-1)(\chi(H)-1)+\sigma(H)       
\end{equation*}
holds, where $\chi(H)$ is the chromatic number of $H$ and $\sigma(H)$ is the size of the smallest colour class for the $\chi(H)$-colouring. In this article, we have studied the \emph{Ramsey goodness} of the graph pair $(C_{m}, \mathbb{K}_{2,n})$, where $\mathbb{K}_{2,n}$ is the complete biparite graph. We have obtained an exact value of $R(\mathbb{K}_{2,n},C_{m})$ for all $n$ satisfying $n\geq 3493$ and $n\geq 2m+499$ where $m\geq 7$ is odd. This shows that $\mathbb{K}_{2,n}$ is $C_{m}$-good, which extends a previous result on the Ramsey goodness of $(C_{m}, \mathbb{K}_{2,n})$. Also, this improves the lower bound on $n$ from a previous result on the Ramsey number $R(B_{n}, C_{m})$
\end{abstract}
\maketitle

\section{Introduction}

Let $G$ and $H$ be two graphs. The \emph{Ramsey number} $R(G,H)$ is the minimum positive integer $n$ such that each red-blue colouring of the edges of the complete graph $\mathbb{K}_{n}$ contains a red copy of $G$ or a blue copy of $H$. Equivalently, it is the minimum integer $n$ such that for any graph $G'$ with $n$ many vertices, either $G'$ contains a copy of $G$ as a subgraph or its complementary graph $\Bar{G'}$ contains a copy of $H$ as a subgraph. Note that, for any two graphs $G$ and $H$, $R(G,H)=R(H,G)$. If $G=\mathbb{K}_{k}$ and $H=\mathbb{K}_{l}$, we simply denote the Ramsey number as $R(k,l)$. Ramsey's theorem states that for any two positive integers $k$ and $l$, $R(k,l)$ exists. The \emph{book graph} of order $(n+k)$ denoted by $B_{n}^{(k)}$ is the graph with $n$ copies of $\mathbb{K}_{k+1}$ sharing a common $\mathbb{K}_{k}$ called \emph{base}. In particular, it is the graph $\mathbb{K}_{k}\circledast\mathbb{K}_{n}^{c}$, where the notation $\circledast$ means that each vertex of $\mathbb{K}_{k}$ is adjacent to all the $n$ isolated vertices. The Ramsey number of book versus itself was first studied by Erd\"{o}s, Faudree, Rousseau, and Schelp \cite{MR479691}. There, they showed that
\begin{equation*}
2^{k}n+o_{k}(n)\leq R(B_{n}^{(k)},B_{n}^{(k)})\leq 4^{k}n. 
\end{equation*}
In addition, Thomason \cite{MR679211} conjectured that the value of $R(B_{n}^{(k)},B_{n}^{(k)})$ is asymptotically equal to the lower bound. The conjecture was proved by Conlon \cite{MR4115773}. In this article, we restrict our discussion to $k=2$ and simply denote the graph by $B_{n}$. 

We discuss two Ramsey numbers, namely $R(B_{n},C_{m})$ and $R(\mathbb{K}_{2,n},C_{m})$. Here $\mathbb{K}_{2,n}$ denotes the complete bipartite graph of order $(n+2)$ with vertex partition $A\sqcup B$ where $|A|=2$ and $|B|=n$. Note that removing the base of $B_{n}$ results in $\mathbb{K}_{2,n}$. In Extremal Combinatorics, the Ramsey number for the book graph versus other types of graphs holds significant importance. One such example is the book-cycle Ramsey number. Significant results have been obtained regarding the book-cycle Ramsey number in recent years. Rousseau and Sheehan \cite{MR486186} proved that $R(B_{n},C_{3})= 2n+3$ for each $n>1$. Faudree et al. \cite{MR1110243} obtained a general result on $R(B_{n},C_{m})$.

\begin{theorem}\cite{MR1110243}
For each integer $n\geq 4m-13$ and odd integer $m\geq 5$, $R(B_{n},C
_{m})=2n+3$, 
\end{theorem}

\noindent In this article, the author also proved that if $n\leq \frac{m}{2}-2$, then $R(B_{n},C_{m})= 2m-1$. Later, the bound of $n$ was improved by Shi \cite{MR2587033}, who proved the result for $n\leq \frac{3m}{2}-\frac{7}{2}$. Lin and Peng \cite{MR4234150} determined the exact values of $R(B_{n},C_{m})$ if $n\in [\frac{2m}{3}-1, \frac{9m}{8}-126]\cap \mathbb{Z}$.
\begin{theorem}\cite{MR4234150}
For each integer $n\geq 1000$, 
\begin{align*}
R(B_{n},C_{m})&=\left\{\begin{array}{lll}
		3m-2& \textnormal{ if  }&  m\in[\frac{8n}{9}+112, n]\cap\mathbb{Z}\\
		3n-1&\textnormal{  if }& m= n+1\\
        3n &\textnormal{  if }& m\in[n+2,\frac{3n+1}{2}]\cap\mathbb{Z}\\
        2m-1&\textnormal{  if }&  m= \left\lceil \frac{3n}{2} \right\rceil +1.
	\end{array}\right.
\end{align*}
\end{theorem}

\noindent In the same article, the authors put down a remark that the value of $R(B_{n},C_{m})$ is still unknown in the other ranges. In a recent paper, Hu et. al.\cite{MR4875667} obtained the exact value of $R(B_{n}, C_{m})$ in such ranges for sufficiently large $n$. The result is as follows.

\begin{theorem}\label{Peng}\cite{MR4875667}
For $m$ odd and $n$ sufficiently large, we have
\begin{align*}
R(B_{n},C_{m})&=\left\{\begin{array}{lll}
		3m-2& \textnormal{ if  }&  n\in[\frac{10m}{9}, \frac{3m-5}{2}]\cap\mathbb{Z}\\
        2n+3 &\textnormal{  if }& n\in[\frac{3m-5}{2},4m]\cap\mathbb{Z}.
    \end{array}\right.  
\end{align*}

\end{theorem}

\noindent In this result, the `largeness' of $n$ arises due to the following result by Haxell, Gould, and Scott\cite{MR1939069}, where they have constructed the constant $K(\epsilon)=\frac{750000}{c^{5}}$. Hence, according to the proof of Theorem~\ref{Peng}, $n$ must satisfy $n\geq \frac{45K}{\epsilon^{4}}=\frac{33750000}{\epsilon^{9}}$ where $\epsilon=\frac{5}{27}$.
\begin{theorem}\label{Haxell et. al.}\cite{MR1939069}
    For each real number $\epsilon\in(0,1)$, there exists a constant $K=K(\epsilon)$ such that if $G$ is graph on $n$ vertices with $n\geq \frac{45K}{\epsilon^{4}}$ and $\delta(G)\geq \epsilon n$, then $G$ contains a cycle of length $l$ for all even $l\in [K-\mathop{ec}(G)]$ and for all odd $l\in[K-\mathop{oc}(G)]$\footnote{Here, $\mathop{ec}(G)$ and $\mathop{oc}(G)$ denote the length of the longest even and odd cycle in $G$ respectively.}.
\end{theorem}

In this article, we investigate the Ramsey number $R(B_{n},C_{m})$ within the range $n\in[2m-3, 4m-14]$, focusing on significantly smaller values of $n$ also. In particular, we have obtained a bound of $R(B_{n}, C_{m})$ that holds for all integers $n$ in the interval $[2m-3, 4m-14]$ where $m\geq 7$ is odd. Consequently, this result applies for all $n\geq 11$ as $m\geq 7$. This is a progress from the previous results. The result is as follows:

\begin{theorem}\label{main theorem 1}
For each odd integer $m\geq 7$ and for each integer $n\geq (2m-3)$,
\begin{equation*}
2n+3\leq R(B_{n}, C_{m}) \leq 2n+\frac{1010}{3}.    
\end{equation*}
\end{theorem}

\noindent However, we believe that the exact value of $R(B_{n},C_{m})$ in this range will be equal to the lower bound under certain conditions on $n$ and $m$. Based on certain observations, we propose the following conjecture. Theorem~\ref{Peng} partially confirms the conjecture for larger values of $n$. We have successfully proved the conjecture for a subgraph $\mathbb{K}_{2,n}$ of $B_{n}$ and will discuss this result later.

\begin{conjecture}
For each integer $n\geq 3493$, $R(B_{n},C_{m})=2n+3$ if $n\geq 2m+499$ and $m\geq 7$ is an odd integer.
\end{conjecture}

Let $G$ be a connected graph and $H$ be any graph. It is well known that 
\begin{equation*}
R(G,H)\geq (|G|-1)(\chi(H)-1)+\sigma(H),
\end{equation*}
where $\chi(H)$ is the chromatic number of $H$ and $\sigma(H)$ is the size of the smallest colour class with respect to $\chi(H)$-colouring. The graph $G$ is called a \emph{$H$-good} graph if the equality holds in the above inequation. Bondy and Erd\"{o}s \cite{MR317991} studied the \emph{Ramsey goodness} of $(C_{m}, \mathbb{K}_{r_{1}, r_{2},\ldots,r_{t}})$. Pokrovskiy and Sudakov\cite{MR4142765} extended this result. They proved the following.

\begin{theorem}\cite{MR4142765}
For each positive integer $r_{1},\ldots r_{t}$, satisfying $r_{t}\geq r_{t-1}\geq \ldots\geq r_{1}$ and $r_{i}\geq i^{22}$ and $m\geq 10^{60}r_{t}$,
\begin{equation*}
R(\mathbb{K}_{r_{1},\ldots,r_{t}}, C_{m})=(t-1)(m-1)+r_{1}.
\end{equation*}
\end{theorem}

\noindent From this result, we deduce the following corollary by plugging in $t=2$; $r_{1}=2$ and $r_{2}=n$. This shows that $C_{m}$ is a $\mathbb{K}_{2,n}$-good if $m$ and $n$ are sufficiently large.

\begin{corollary}
For each positive integer $m$ and $n$ with $n\geq 2^{22}$ and $m\geq 10^{60}n$, $R(\mathbb{K}_{2,n},C_{m})= m+1$. 
\end{corollary}

\noindent Note that, in the above result, the value of $R(\mathbb{K}_{2,n}, C_{m})$ depends on the parameter $m$ only. But in this article, we have obtained an exact value of the same involving the parameter $n$. An odd cycle has a chromatic number $3$ with the smallest colour class of size $1$. Therefore, if $m$ is an odd integer and if the equality $R(\mathbb{K}_{2,n},C_{m})=2n+3$ holds then $\mathbb{K}_{2,n}$ is $C_{m}$-good. From Theorem~\ref{main theorem 1}, we have observed that if $m$ and $n$ lie within a specific range, then the deletion of the edge from the base of $B_{n}$ helps to find the exact value of $R(\mathbb{K}_{2,n},C_{m})$. Consequently, we conclude that $\mathbb{K}_{2,n}$ is $C_{m}$-good. It is worth noting that the Ramsey goodness property of $(C_{m},\mathbb{K}_{2,n})$ is getting exchanged (i.e., $\mathbb{K}_{2,n}$ is $C_{m}$-good) with the variation of the parameters $n$ and $m$. The result we obtain is as follows:

\begin{theorem}\label{main theorem 2}
For each odd integer $m\geq 7$ and for each integer $n\geq2m+499$ and $n\geq 3493$,  
\begin{equation*}
 R(\mathbb{K}_{2,n},C_{m})=2n+3.    
\end{equation*}
\end{theorem}
\noindent It is important to note that, since $\mathbb{K}_{2,n}$ is a subgraph of $B_{n}$, the exact value of $R(\mathbb{K}_{2,n},C_{m})$ can be obtained directly from Theorem~\ref{Peng} under the same condition on $n$ and $m$. In particular, $R(\mathbb{K}_{2,n},C_{m})=2n+3$ if $n\in [\frac{3m-5}{2},4m]$, $m$ odd and $n\geq \frac{33750000}{\epsilon^{9}}$. However, in our result, we have established the same value for a significantly smaller lower bound on $n$.

\section{Preliminary Results}
Let $G=(V, E)$ be a graph with vertex set $V(G)$ and edge set $E(G)$. The order of the graph is the number of vertices in it and is denoted by $|V(G)|$. For a vertex $v\in V(G)$, $N_{G}(v)$ denotes the set of neighbours (i.e., the vertices adjacent to it) of $v$ in $G$. For a graph $G$, minimum and maximum degrees are denoted by $\delta (G)$ and $\Delta (G)$ respectively. The \emph{circumference} and \emph{girth} of a graph $G$ refer to the lengths of the longest and shortest cycles, respectively. The circumference is denoted by $\mathop{c}(G)$ and the girth is denoted by $\mathop{g}(G)$. The \emph{connectivity} of a graph $G$ is the minimum number of vertices that need to be deleted so that the graph becomes disconnected. It is denoted by $\mathop{k}(G)$. A connected graph is said to be \emph{$k$-connected}, where $k\geq2$ is an integer, if deletion of each subset $S$ of $V(G)$ with $|S|<k$ does not make the graph disconnected. Thus, if $G$ is $k$-connected, then $\mathop{k}(G)\geq k$. Conversely, if $\mathop{k}(G)\geq k$, and if $G$ is not $k$-connected, then there exists a subset $S$ of $V(G)$ with $|S|<k$, such that deletion of $S$ makes the graph disconnected. Hence $\mathop{k}(G)\leq|S|$, which contradicts our assumption that $\mathop{k}(G)\geq k$. This means if $\mathop{k}(G)\geq k$, then $G$ is $k$-connected. Therefore, $G$ is $k$-connected if and only if $\mathop{k}(G)\geq k$. For a graph $H$ and a positive integer $n$, \emph{extremal number}, denoted by $\ex(n, H)$, is the maximum number of edges of a graph on $n$ vertices which is $H$-free. This implies that any graph on $n$ vertices with at least $\ex(n,H)$ edges contains a copy of $H$. Throughout this article, we consider only simple graphs.

For a graph $G$, the lower bound of its circumference depends on the minimum degree. Imposing additional criteria of ``$2$-connectivity" on the graph $G$ leads to an improved bound. The following result is due to Dirac. 

\begin{lemma}\cite{MR47308}\label{Dirac}
For a graph $G$, if each vertex has degree $d$, then it has a cycle of length at least $d+1$. If $G$ is $2$-connected, then it contains a cycle of length at least $2d$.
\end{lemma}

\noindent The above result shows that if a graph $G$ has minimum degree $\delta(G)$, then $\mathop{c}(G)\geq \delta(G)+1$. However, in our proof, we use $\mathop{c}(G)\geq \delta(G)$ for the simplicity of calculation. The following result suggests that a graph with a higher minimum degree cannot possess a girth greater than 4. In \cite{MR4234150}, the author mentioned the following lemma. They omitted the proof and mentioned it as a simple fact which can be checked using Breadth-First-Search.

\begin{lemma}\cite{MR4234150}\label{fact check}
For each $\epsilon>0$, if $G$ is a graph with $|V(G)|>\frac{1}{\epsilon^2}$ and $\frac{\delta(G)}{|V(G)|}\geq \epsilon$, then $\mathop{g}(G)\leq 4$.
\end{lemma}

\noindent However, we establish the above result for graphs satisfying $|V(G)|>\frac{10}{\epsilon^{2}}$ by using the \emph{dependent random choice} technique. We briefly describe the proof here. The proof relies on the following auxiliary results by Fox and Sudakov \cite{MR2768884}.

\begin{lemma}\cite[Lemma~2.1]{MR2768884}\label{Existence of a rich set}
Let $a, m, r$ be positive integers. $G=(V,E)$ be a graph on $n$ vertices with average degree $d= \frac{2|E|}{n}$. If there exists a positive integer $t$ such that $n(\frac{d}{n})^{t}-\binom{n}{r}(\frac{m}{n})^{t}\geq a$, then there exists a subset $U\subseteq V$ with $|U|\geq a$ such that for each $Q\in \binom{U}{r}$, $|N(Q)|\geq m$.
\end{lemma}

\begin{lemma}\cite[Lemma~3.2]{MR2768884}\label{embedding of bipartite graph}
Let $H$ be a bipartite graph with partition $A\sqcup B$ where $|A|=a, |B|=b$ and $d(v)\leq r$ for all $v\in B$. If $G$ is a graph with a subset $U\subseteq V(G)$ such that $|U|\geq a$ and for all $S\in \binom{U}{r}$, $|N(S)|\geq (a+b)$, then $G$ contains a copy of $H$.
\end{lemma}

\begin{lemma}\label{fact checked}
For each $\epsilon>0$, if $G$ is a graph on the the $n$ vertices with $\delta(G)\geq \epsilon n$ and $n\geq \frac{10}{\epsilon^{2}}$, then $\mathop{g}(G)\leq 4$.
\end{lemma}
\begin{proof}
Consider the graph $H=\mathbb{K}_{2,2}$. From Theorem~\ref{Existence of a rich set}, we plug in $a=b=2$, $m=a+b=4$, and $t=2$. We first prove that $\ex(n, \mathbb{K}_{2,2})\leq cn^{\frac{3}{2}}$ for $c=\sqrt{\frac{5}{2}}$. Let $G_{1}$ be a graph on $n$ vertices with at least $cn^{\frac{3}{2}}$ many edges. If $d$ is the average degree of $G_{1}$, then it satisfies $d=\frac{2|E(G)|}{n}\geq \frac{2cn^{\frac{3}{2}}}{n}= 2cn^{\frac{1}{2}}$. Lemma~\ref{embedding of bipartite graph} implies that, to find an embedding of $\mathbb{K}_{2,2}$ in $G_{1}$ we need to show the existence of an $U\subseteq V(G_{1})$ with $|U|\geq 2$ such that for all $Q\in \binom{U}{2}$, $|N(Q)|\geq 4$. Theorem~\ref{Existence of a rich set} ensures the existence of such a set $U$ if 
\begin{equation*}
n\frac{d^{t}}{n^{t}}- \binom{n}{r}\left(\frac{m}{n}\right)^{t}\geq a.
\end{equation*}
In our case, we have $(2c)^{2}-\binom{n}{2}(\frac{4}{n})^{2}\geq (2c)^{2}-\frac{n^{2}}{2}\frac{16}{n^{2}}\geq 2$ which implies $c\geq \sqrt{\frac{5}{2}}$. Here we consider the minimum value of $c$ which proves that  $\ex(n, \mathbb{K}_{2,2})\leq \sqrt{\frac{5}{2}}n^{\frac{3}{2}}$. To complete the proof, we let $\epsilon>0$. Consider a graph $G$ on $n$ vertices with $\delta(G)\geq \epsilon n$ and $n\geq \frac{10}{\epsilon ^{2}}$ and consequently $\sqrt{n}\geq \frac{\sqrt{10}}{\epsilon}$. Then $|E(G)|\geq \frac{\epsilon n^{2}}{2}= \frac{\epsilon \sqrt{n}n^{\frac{3}{2}}}{2}\geq \sqrt{\frac{5}{2}}n^{\frac{3}{2}}\geq\ex(n, \mathbb{K}_{2,2})$. The last inequality follows from the given condition that $n\geq \frac{10}{\epsilon ^{2}}$ which implies $\sqrt{n}\geq \frac{\sqrt{10}}{\epsilon}$. Thus, $G$ contains a $4$-cycle and consequently $\mathop{g}(G)\leq 4$.  
\end{proof}

In determining the Ramsey number involving cycles, the property of pancyclicity is an important tool. In certain instances, demonstrating the existence of a cycle of a specific length within a graph can be challenging. To deal with such cases, the pancyclicity property provides valuable assistance. A graph $G$ is said to be \emph{pancyclic} if it contains cycles of all lengths $l$ between $3$ and $|V(G)|$. It is \emph{weakly pancyclic} if it contains cycles of all lengths between girth and circumference. The following result shows that the graph with a high minimum degree contains cycles of all lengths. 

\begin{lemma}[Bondy \cite{MR285424}]\label{pancyclic lemma 1}
For each graph $G$, if $\delta(G) \geq \frac{|V(G)|}{2}$, then such $G$ is either a pancyclic graph or $G$ isomorphic to the complete bipartite graph $\mathbb{K}_{r,r}$, where the positive integer $r=\frac{|V(G)|}{2}$.
\end{lemma}

\noindent In the article \cite{MR1611825}, the authors studied the properties of weakly pancyclic graphs. The following two results are some of the most important tools in our proof.    

\begin{lemma}[Brandt et al.\cite{MR1611825}]\label{Brandt 2}
For each non-bipartite graph $G$, if $\delta(G)\geq \frac{|V(G)|+2}{3}$, then such $G$ is a weakly pancyclic graph with girth $3$ or $4$.
\end{lemma}

\begin{lemma}[Brandt et al. \cite{MR1611825}]\label{Brandt}
Let $G$ be a $2-$connected non-bipartite graph. If $\delta(G)\geq \frac{|V(G)|}{4}+250$, then $G$ is weakly pancyclic unless it has odd girth $7$, in which case $G$ contains a copy of the cycle of each length $l$, where $l\neq5$ and $4\leq l\leq\mathop{c}(G)$.
\end{lemma}

A graph $G$ is \emph{panconnected} if for each $x,y \in V(G)$, there exists an $x$-$y$ path (i.e., a path from vertex $x$ to vertex $y$) of (each) length $l$, where $l\in[2,|V(G)|-1]\cap\mathbb{Z}$. Let $G$ be a bipartite graph with partition $A\sqcup B$ with $|A|\geq |B|$. If for each $x,y \in V(G)$, there exists an $x$-$y$ path of length $l$ for all ``possible" $l\in [2, 2|B|]\cap \mathbb{Z}$, then such $G$ is called a \emph{bipanconnected} graph. More precisely, in a bipanconnected graph $G$, for each $x,y\in V(G)$, there exists an $x$-$y$ path of each length $l$ where:-
\begin{align*}
l\leq\left\{\begin{array}{ll}
		2|B|-1& \textnormal{ and } l\in2\mathbb{Z}+1 \textnormal{ if }  x\in A, y\in B\\
		2|B|-2& \textnormal{ and }l\in2\mathbb{Z} \textnormal{ if } x, y\in B\\
        2|B|& \textnormal{ and }l\in2\mathbb{Z}  \textnormal{ if } x, y\in A\\  
	\end{array}\right..        
\end{align*}
The bipanconnectedness property is the key tool that is being used in the proof of Theorem~\ref{main theorem 1}. We borrow the above description of bipanconnectedness from \cite{MR3713394}.

\begin{lemma}[Du et al.\cite{MR3713394}]\label{Hui}
Let $G$ be a bipartite graph with partition $A\sqcup B$, where $|A|\geq|B|\geq 2$. If $\delta(G) \geq \frac{|A|}{2} + 1$, then such $G$ is bipanconnected.
\end{lemma}

Here, we prove some lemmas which have been used to prove the main results. The results are the following.
\begin{lemma}\label{intersection lemma}
A graph $G$ is $B_{n}-$free if and only if $|N_{G}(x)\cap N_{G}(y)|\leq n-1$ for all $x,y \in V(G)$ with $\{x,y\}\in E(G)$.
\end{lemma}
\begin{proof}

Let $G$ do not contain any copy of $B_{n}$, and there exist $x,y$ with $\{x,y\}$ being an edge such that $|N_{G}(x)\cap N_{G}(y)|\geq n$. Let $H=N_{G}(x)\cap N_{G}(y)$. Then the subgraph of $G$ with vertex set $V(H)\sqcup \{x,y\}$ contains a copy of $B_{n}$. This contradicts the assumption that $G$ is $B_{n}$-free. Similarly, let $|N_{G}(x)\cap N_{G}(y)|\leq n-1$ for all $x,y$ with $\{x,y\}\in E(G)$. If $G$ contains a copy of $B_{n}$, then there exist two vertices $a,b\in V(G)$ with at least $n$ many common neighbours, i.e., $|N_{G}(a)\cap N_{G}(b)|\geq n$, which is a contradiction. This completes the proof.
\end{proof}

\noindent From the above proof, we observe the following corollary. 

\begin{corollary}\label{intersection lemma 2}
A graph $G$ is $\mathbb{K}_{2,n}-$free if and only if $|N_{G}(x)\cap N_{G}(y)|\leq (n-1)$ for all $x,y \in V(G)$.
\end{corollary}

\begin{lemma}\label{maximum degree lemma}
Let $G$ be a graph with $|V(G)|\geq 2n+3$. If such a graph is $B_{n}-$free and its complementary graph $\Bar{G}$ is $C_{m}-$free, where $m\leq 2n+2$, then $\Delta(G)<2n+2$.
\end{lemma}

 \begin{proof}
Let $G$ be a graph with $|V(G)|\geq2n+3$. If $\Delta(G)\geq 2n+2$. We choose $v\in V(G)$ to be the vertex with the maximum degree. Consider $H=G[X]$, the subgraph induced by $X$ where $X\subseteq N_{G}(v)$ with $|X|=(2n+2)$. Since $G$ is  $B_{n}$-free, using Lemma~\ref{intersection lemma}, we conclude that $|N_{H}(x)|\leq (n-1)$ for all $x\in V(H)$. Therefore, $\Delta(H)\leq (n-1)$. Since $\Delta(H)+\delta(\Bar{H})= |V(H)|-1= 2n+1$, we have $\delta (\Bar{H})\geq n+2$.
    
\noindent{\textsl{Claim} :} The graph $\bar{H}$ is a non-bipartite  graph.
\begin{proof}[\tt{Proof of claim} :]\renewcommand{\qedsymbol}{}
Let $\bar{H}$ be a bipartite graph with the partition $V(\Bar{H})= A\sqcup B$. Here, for each $x\in A$, $N_{\Bar{H}}(x)\subset B$. Therefore, $|B|\geq|N_{\Bar{H}}(x)|\geq\delta(\Bar{H})\geq n+2$. Similarly, $|A|\geq\delta(\Bar{H})\geq n+2$. Consequently, $|V(\Bar{H})|=|A|+|B|\geq 2(n+2)$, which contradicts $|V(H)|=2n+2$. 
\end{proof}

Therefore, $\Bar{H}$ is not a bipartite graph with $\delta(\Bar{H})\geq (n+2)>(n+1)= \frac{V(\Bar{H})}{2}$. Using Lemma~\ref{pancyclic lemma 1}, we have that $\Bar{H}$ is a pancyclic graph. Here $m\leq2n+2=|V(\Bar{H})|$. Thus, the subgraph $\Bar{H}$ of the graph $\Bar{G}$ contains a copy of $C_{m}$. This contradicts the assumption that $\Bar{G}$ is $C_{m}-$free.
 \end{proof}

\noindent From the above proof, we have a similar corollary. 

\begin{corollary}\label{maximum degree lemma 2}
Let $G$ be a graph with $ 2n+3$ many vertices such that $G$ is $\mathbb{K}_{2,n}$-free and $\Bar{G}$ is $C_{m}$-free, where $m\leq 2n+2$. Then $\Delta(G)< 2n+2$, where $\Delta(G)$ denotes the maximum degree of the graph.
\end{corollary}

\section{Proof of Theorem~\ref{main theorem 1}}

Let $n$ and $m$ be a positive integer and odd positive integer, respectively, such that $n\geq (2m-3)$ where $m\geq 7$, i.e., $m\in\left[7,\frac{n+3}{2}\right]\cap(2\mathbb{Z}+1)$. We consider the graph $G$ to be mutually disjoint copies of two $\mathbb{K}_{n+1}$. Then $\Bar{G}$ equals the bipartite graph $\mathbb{K}_{n+1,n+1}$. Here $G$ does not contain any $B_{n}$ as it is a disjoint union of two $\mathbb{K}_{n+1}$ and $B_{n}$ is a graph of order $(n+2)$. Furthermore, since $m$ is an odd integer, $\Bar{G}$ does not contain a copy of $C_{m}$. Thus, this graph $G$ serves as an example of a graph that is $B_{n}$-free, and its complementary graph $\Bar{G}$ is $C_{m}$-free.  Hence $R(B_{n},C_{m})>2n+2$ and consequently, $R(B_{n},C_{m})\geq2n+3$.

To prove $R(B_{n}, C_{m}) \leq 2n+\frac{1010}{3}$, we use the method of contradiction. So, let there exist a graph $G$ with $2n+\frac{1010}{3}= 2(n+1)+\frac{1004}{3}$ \footnote{Our primary motive here is to apply the Lemma~\ref{Brandt} on the graph $\Bar{G}$. A sufficient condition ``$V(\Bar{G})\geq2(n+1)+\frac{1004}{3}$" may fulfill our motive.} many vertices such that $G$ contains no copy of $B_{n}$ (i.e., $G$ is $B_{n}$-free) and its complementary graph $\Bar{G}$ contains no copy of $C_{m}$ (i.e., $\Bar{G}$ is $C_{m}$-free). We divide the proof into the following two cases.

\noindent{\textbf{Case~I :}} Let $\Delta(G)<\frac{3(n+1)}{2}$. 

Here since $\Delta(G)+\delta(\Bar{G})=|V(G)|-1$ and $|V(G)|=|V(\Bar{G})|=2(n+1)+\frac{1004}{3}$ holds, we have
\begin{equation*}
\delta(\Bar{G})>\frac{(n+1)}{2}+\frac{1001}{3}=\frac{|V(\Bar{G})|}{4}+250.   
\end{equation*}

\begin{lemma}\label{non-bipartite}
Let $G$ be a graph with $2(n+1)+\frac{1004}{3}$ many vertices such that $G$ is $B_{n}$-free. If for such a graph $\Delta(G)<\frac{3(n+1)}{2}$ holds, then the complementary graph $\Bar{G}$ is non-bipartite.
\end{lemma}
\begin{proof}
If not, then let $\Bar{G}$ be a bipartite with parts $A$ and $B$, where $|A|\geq|B|$. Since 
\begin{equation*}
|A|+|B|=|V(\Bar{G})|=2(n+1)+\frac{1004}{3}    
\end{equation*}
holds we have $2|A|\geq 2(n+1)+\frac{1004}{3}$, i.e., $|A|\geq (n+1)+\frac{502}{3}>(n+2)$ holds. Since all the vertices in $A$ are isolated in $\Bar{G}$, it follows that the graph $G$ contains a copy of $\mathbb{K}_{|A|}$. Since $|A|>(n+2)$, and $B_{n}$ is a subgraph of $\mathbb{K}_{n+2}$, such $G$ contains a copy of $B_{n}$, which contradicts the assumption that $G$ is $B_{n}$-free. 
\end{proof}

Using Lemma~\ref{non-bipartite}, we have the graph $\Bar{G}$ is non-bipartite, and we further study the graph parameter $\mathop{k}(\Bar{G})$. Here we inquire into the cases $\mathop{k}(\Bar{G})\geq2$; $\mathop{k}(\Bar{G})=1$ and $\mathop{k}(\Bar{G})=0$. Each such case leads us to a contradiction.

If $\mathop{k}(\Bar{G})\geq 2$ (Hence $\Bar{G}$ is $2-$connected.), then using Lemma~\ref{Brandt} we conclude that $\Bar{G}$ is a weakly pancyclic graph or contains a copy of a cycle of each length between $4$ and $\mathop{c}(\Bar{G})$ except $5$.
Let $\Bar{G}$ be weakly pancyclic. Using Lemma~\ref{Dirac}, we have  
\begin{equation*}
m\leq\frac{n+3}{2}<\left(\frac{n+1}{2}+\frac{1001}{3}\right)\leq\delta(\Bar{G})\leq c(\Bar{G}).   
\end{equation*}

\noindent Since $\delta(\Bar{G})\geq \frac{|V(\Bar{G})|}{4}+250> \frac{|V(\Bar{G})|}{4}$, applying $\epsilon=\frac{1}{4}$ in Lemma~\ref{fact check}\footnote{Note that, if we apply Lemma~\ref{fact checked}, then an extra condition of $n\geq 160$ is needed in Theorem~\ref{main theorem 1}.}, we have $\mathop{g}(\Bar{G})\leq 4$. Thus, the integer $m$ satisfies $\mathop{g}(\Bar{G})\leq m \leq \mathop{c}(\Bar{G})$. This follows that $\Bar{G}$ contains a copy of $C_{m}$. In the other case, if $\Bar{G}$ contains a copy of a cycle of each length between $4$ and $\mathop{c}(\Bar{G})$ except $5$, then in particular $\Bar{G}$ contains a copy of $C_{m}$, since $m\geq 7$. From both cases, we conclude that $\Bar{G}$ contains a copy of $C_{m}$, which is a contradiction. 

If $\mathop{k}(\Bar{G})= 1$, then there exists a vertex $v$ such that $\Bar{G}-\{v\}$ is disconnected. Here $\Bar{G}-\{v\}$ denotes the graph with deleted vertex $v$ (hence all edges containing $v$) from the graph $\Bar{G}$. Suppose $\Bar{G}-\{v\}$ contains $r$ many distinct components $G_{1},\ldots,G_{r}$ (say) for some integer $r\geq2$. Note that, for each $i\in[r]$,
\begin{equation*}
 |V(G_{i})|\geq\delta(G_{i})\geq \delta(\Bar{G}-\{v\}) \geq \delta(\Bar{G})-1 \geq \frac{n+1}{2}+\frac{998}{3}. 
\end{equation*}
This implies that 
\begin{equation*}
2(n+1)+\frac{1001}{3}=|V(\Bar{G}-\{v\})|=|V(G_{1})|+\ldots+|V(G_{r})|\geq r\left(\frac{n+1}{2}+\frac{998}{3}\right),
\end{equation*} 
which implies $r\leq3$. If $r=2$, then without loss of generality we assume that $|V(G_{1})|\geq |V(G_{2})|$. Here we see that 
\begin{equation*}
 2(n+1)+\frac{1001}{3}=|V(\Bar{G}-\{v\})|=|V(G_{1})|+|V(G_{2})|\geq 2|V(G_{2})|   
\end{equation*}
This means $|V(G_{2})|\leq (n+1)+\frac{1001}{6}$. Also here, we have  
\begin{equation*}
\delta(G_{2})\geq\delta(\Bar{G})-1\geq\frac{n+1}{2}+\frac{998}{3}> \frac{|V(G_{2})|}{2}.     
\end{equation*}
Hence, using the lemma~\ref{pancyclic lemma 1}, $G_{2}$ is either a pancyclic graph or it is isomorphic to the complete bipartite graph $\mathbb{K}_{s,s}$ for some integer $s$ satisfying $s\geq \delta(G_{2})\geq \frac{n+1}{2}+\frac{998}{3}$. If $G_{2}$ is isomorphic to the complete bipartite graph $\mathbb{K}_{s,s}$, then
\begin{equation*}
 (n+1)+\frac{1001}{6}\geq|V(G_{2})|=2s\geq 2\left(\frac{n+1}{2}+\frac{998}{3}\right)=(n+1)+\frac{1996}{3}   
\end{equation*}
holds, which is a contradiction. Hence, $G_{2}$ is a pancyclic graph. Since $m\leq\frac{n+3}{2}\leq\frac{n+1}{2}+\frac{998}{3}\leq|V(G_{2})|$ holds, we have that $G_{2}$ contains a copy of $C_{m}$. In particular, $\Bar{G}$ contains a copy of $C_{m}$, which is a contradiction. So let $r=3$ and without loss of generality, we assume that $|V(G_{1})|\geq |V(G_{2})|\geq |V(G_{3})|$. Thus
\begin{equation*}
 |V(\Bar{G}-\{v\})|= 2(n+1)+\frac{1001}{3}=|V(G_{1})|+|V(G_{2})|+|V(G_{3})| \geq 3|V(G_{3})|,   
\end{equation*}
consequently, $|V(G_{3})|\leq \frac{2(n+1)}{3}+\frac{1001}{9}$. Hence, using similar arguments, we reach a contradiction that the subgraph $G_{3}$ of the graph $\Bar{G}$ contains a copy of $C_{m}$.

The remaining case is $\mathop{k}(\Bar{G})=0$. Here $\Bar{G}$ is disconnected. On assuming $\Bar{G}$ contains $r$ many distinct components $G_{1},\ldots,G_{r}$ and arguing similarly, we reach the same contradiction that $\Bar{G}$ contains a copy of $C_{m}$.

\noindent{\textbf{Case II :}}  Let $\Delta(G)\geq \frac{3(n+1)}{2}$. i.e., there exists a vertex $v$ such that $|N_{G}(v)|=\Delta(G)=\frac{3(n+1)}{2}+p$ for some integer $p\geq 0$. Here $N_{G}(v)=\{x\in V(G): \{x,v\}\in E(G)\}$, denotes the \emph{neighbours} of the vertex $v$ in the graph $G$.  

Here $|V(G)|=2(n+1)+\frac{1004}{3}>2n+3$. We assumed that $G$ is $B_{n}-$free and $\Bar{G}$ is $C_{m}-$free, where $m\leq 2n+2$. Thus, using Lemma~\ref{maximum degree lemma}, we have $\Delta(G)<2n+2$; this implies $p\leq \frac{n+1}{2}-1$. Now we construct the graph $H= G[M]$ where $M\subseteq N_{G}(v)$ with $|M|=\frac{3(n+1)}{2}$. Here such graph $H$ contains $\frac{3(n+1)}{2}$ many vertices and the edge set $E(H)=\{\{x,y\}\in E(G):x,y\in V(H)\}$.  

We further study the graph $\Bar{H}$. Here we inquire into the cases where the graph $\Bar{H}$ is a bipartite graph or not a bipartite graph. Each such case leads us to a contradiction. Using Lemma~\ref{intersection lemma}, we have $|N_{H}(u)|\leq (n-1)$ for all $u\in V(H)$. Hence $\Delta(H)\leq (n-1)$. Consequently,
\begin{equation*}
\delta(\Bar{H}) \geq (|V(\Bar{H})| - 1)- \Delta(\Bar{H}) \geq\left(\frac{3(n+1)}{2}-1\right)-(n-1)=\frac{n+1}{2}+1\geq \frac{|V(\Bar{H})|+2}{3}.   
\end{equation*}

\noindent If $\Bar{H}$ is not a bipartite graph, then using Lemma~\ref{Brandt 2}, we have the graph $\Bar{H}$ is a weakly pancyclic graph with girth $3$ or $4$. Also using the Lemma~\ref{Dirac}, we have $\mathop{c}(\Bar{H})\geq \delta(\Bar{H})\geq (\frac{n+1}{2}+1)\geq m$. Here 
\begin{equation*}
 m\in\left[7,\frac{n+3}{2}\right]\cap(2\mathbb{Z}+1)\subset[g(\Bar{H}),c(\Bar{H})]\cap\mathbb{Z}. 
\end{equation*}
Hence the subgraph $\Bar{H}$ of the graph $\Bar{G}$ contains a copy of $C_{m}$, which is a contradiction.

However, if $\Bar{H}$ is a bipartite graph with the partition $V(\Bar{H})=A\sqcup B$, where $|A|\geq |B|$. Note that $\frac{3(n+1)}{2}= |A|+|B|\geq 2|B|$, which implies $|A|\geq\frac{3(n+1)}{4}\geq |B|$. Consequently, we have 
\begin{equation*}
 |A|\geq\frac{3(n+1)}{4}\geq |B|\geq \delta(\Bar{H}) \geq \frac{n+1}{2}+1    
\end{equation*}
Since each vertex of $A$ is an \emph{isolated vertex} in the complementary graph $\Bar{H}$, and $v$ is adjacent to each vertex of $A$ in the graph $G$, we have $A\sqcup \{v\}$ forms a copy of complete graph $\mathbb{K}_{|A|+1}$ in $G$. Hence $|A|\leq n$; otherwise, $A\sqcup \{v\}$ contains a copy of $\mathbb{K}_{n+2}$. In particular, $G$ contains a copy of $B_{n}$, which is a contradiction. Let $|A|= (n-k)$ for some integer $k\geq 0$. Since $|A|\geq \frac{3(n+1)}{4}$, we have $k\leq \frac{n+1}{4}-1$. Then $|B|=(\frac{n+3}{2}+k)$. We construct 
\begin{equation*}
 X=\left\{x\in V(G)| x\notin V(H), x\neq v \right\}.   
\end{equation*}
Since $|V(G)|=|X|+|\{v\}|+|V(H)|$, we have $|X|= (2(n+1)+\frac{1001}{3})-(\frac{3(n+1)}{2}+1)= \frac{n+1}{2}+\frac{998}{3}$. This means $X$ is non-empty and we claim the following.

\noindent{\textsl{Claim} :} There exists $x\in X$ such that $\{x,a\}\in E(\Bar{G})$ and $\{x,b\}\in E(\Bar{G})$ for some $a\in A$ and $b\in B$.

\begin{proof}[\tt{Proof of claim} :]
Suppose on the contrary no such $x$ exists. This implies that for each $x\in X$, in $\Bar{G}$ it is either non-adjacent to all of $A$ or all of $B$ or all of $A, B$ both. Equivalently, each $x\in X$ is adjacent to all of $A$ or all of $B$ or all of $A$ and $B$ both in $G$. We construct the following sets
\begin{align*}
 &X_{A}=\left\{x\in X| \{x,a\}\in E(G) \textnormal{ for each } a \in A \textnormal{ and } \{x,b\}\notin E(G) \textnormal{ for some } b\in B\right\}\\
&X_{B}= \{x\in X| \{x,b\}\in E(G) \textnormal{ for each } b \in B \textnormal{ and } \{x,a\}\notin E(G) \textnormal{ for some } a\in A\}\\
 &X_{AB}=\left\{x\in X| \{x,a\} \in E(G) \textnormal{ and } \{x,b\}\in E(G) \textnormal{ for each } a \in A,b\in B\right\}
\end{align*}
 
Note that, $X_{A}\sqcup X_{B}\sqcup X_{AB}\subseteq X$. From our assumption, we also have that $X \subseteq X_{A}\sqcup X_{B}\sqcup X_{AB}$. Hence, we conclude,
\begin{equation*}
 X= X_{A}\sqcup X_{B}\sqcup X_{AB}.   
\end{equation*}
If $|X_{A}|\geq (k+1)$, then
\begin{equation*}
 |A\sqcup X_{A}\sqcup \{v\}|=|A|+|X_{A}|+1\geq n-k+k+1+1=n+2.   
\end{equation*}
Here we construct a copy of $B_{n}$ in the graph $G$ using the vertices of $X_{A}\sqcup A\sqcup\{v\}$. We choose $a,a'\in A$. Since $A$ consists of isolated vertices in the complementary graph $\Bar{G}$, we have  $\{a,a'\}\in E(G)$. 
Note that each vertex $\alpha\in A$ is adjacent to all the vertices of $(A\smallsetminus\{\alpha\})\sqcup X_{A}\sqcup \{v\}$ in $G$. By choosing the common neighbours of $a$ and $a'$ in $G$, we get the required copy of $B_{n}$. More precisely if $\beta\in (A\smallsetminus\{a,a'\})\sqcup X_{A}\sqcup\{v\})$, then such $\beta\in N_{G}(a)\cap N_{G}(a')$. This is a contradiction to the assumption.  

If $|X_{A}|< (k+1)$, then 
\begin{equation*}
 |X_{B}\sqcup X_{AB}|= |X|-|X_{A}|\geq\left(\frac{n+1}{2}+\frac{998}{3}-k\right)>\left(\frac{n-3}{2}+2-k\right).   
\end{equation*}
Using a similar argument, we get a copy of $B_{n}$ inside $(B\sqcup X_{B}\sqcup X_{AB}\sqcup \{v\})$ in the graph $G$ and conclude the contradiction to the assumption. Hence, the claim is established.  
\end{proof}

Using the claim above, we choose one such $x$ and the corresponding $a\in A$ and $b\in B$ such that $\{x,a\}\in E(\Bar{G})$ and $\{x,b\}\in E(\Bar{G})$. Since $\delta(\Bar{H}) \geq \frac{(n+1)}{2}+1$ and $|A|\leq n$, we have $\delta(\Bar{H})\geq \frac{|A|}{2}+1$. Using Lemma~\ref{Hui}, we have that the graph $\Bar{H}$ is bipanconnected. We recall that our assumption on $m$ is $m\in[7,\frac{n+3}{2}]\cap(2\mathbb{Z}+1)$. Since $|B|\geq \frac{(n+1)}{2}+1$, $(2|B|-1)\geq (n+2)$, we have 
\begin{equation*}
 m-2\in\left[5,\frac{n-1}{2}\right]\cap(2\mathbb{Z}+1)\subset[2,n+2]\cap(2\mathbb{Z}+1)\subset[2,2|B|-1]\cap(2\mathbb{Z}+1).
\end{equation*}
This implies that there exists an $a$-$b$ path $P$ (say) of length $(m-2)$ in $\Bar{H}$. Now $P$ together with the vertex $x$ form a cycle $C_{m}$ in the graph $\Bar{G}$. This contradicts the assumption.

Finally, we conclude that $2n+3\leq R(B_{n}, C_{m}) \leq 2n+\frac{1010}{3}$.

\section{Proof of Theorem~\ref{main theorem 2}}
We let $n$ and $m$ be a positive integer and an odd positive integer, respectively, such that $n\geq (2m+499)$ (i.e., $\frac{n+1}{2}-250 \geq m$) where $n\geq 3493$ and $m\geq 7$. We aim to show that $R(\mathbb{K}_{2,n}, C_{m}) >2n+2$. Here also, we consider the same graph as the previous, i.e., $G$ be such that $\Bar{G}$ is isomorphic to $\mathbb{K}_{n+1,n+1}$. Due to similar arguments in the proof of Theorem~\ref{main theorem 1}, we have $R(\mathbb{K}_{2,n}, C_{m})> 2n+2$. To prove $R(\mathbb{K}_{2,n}, C_{m}) \leq 2n+3$, we employ a proof by contradiction. In this proof, Case I is the same as before. The only differences are the parameters and numbers. Hence, we skip some similar computations in this proof. In case II, this proof differs from the previous theorem in argument. Unless mentioned, ``similar argument" means similar arguments as in the proof of Theorem~\ref{main theorem 1}. We let there exist a graph $G$ on $2n+3$ many vertices such that $G$ contains no  $\mathbb{K}_{2,n}$ and $\Bar{G}$ does not contain any  $C_{m}$. We divide the proof into the following cases.

\noindent\textbf{Case I :} Let $\Delta(G)<\frac{3(n+1)}{2}-250$. This implies $\delta(\Bar{G})\geq \frac{(n+1)}{2}+251$. Note that here $|V(G)|=|V(\Bar{G})|= 2n+3 = 2(n+1)+1$, thus $\delta(\Bar{G})\geq \frac{(|V(\Bar{G})|)}{4}+250$. Using similar arguments, as in Lemma~\ref{non-bipartite}, here we have the following lemma.

\begin{lemma}
The complementary graph $\Bar{G}$ is non-bipartite.
\end{lemma}

\noindent Like before, we consider the three cases $\mathop{k}(\Bar{G})\geq 2$, $\mathop{k}(\Bar{G})=1$ and $\mathop{k}(\Bar{G})=0$ and reach a contradiction in each cases.

If $\mathop{k}(\Bar{G})\geq 2$ (i.e., $\Bar{G}$ is $2-$connected) then from Lemma~\ref{Brandt} we conclude that $\Bar{G}$ is either a weakly pancyclic graph or contains all cycles of each length between $4$ and $\mathop{c}(\Bar{G})$ except $5$. Lemma~\ref{Dirac} implies that 
\begin{equation*}
\mathop{c}(\Bar{G})\geq \delta(\Bar{G}) \geq\left(\frac{n+1}{2}+251\right)> \frac{n+1}{2}-250 \geq m.
\end{equation*}
Using a similar argument, we conclude that $\Bar{G}$ contains a $C_{m}$ which is a contradiction.

If $\mathop{k}(\Bar{G})= 1$, then there exists a vertex $v$ such that $\Bar{G}-\{v\}$ is disconnected. Hence $\delta(\Bar{G}-\{v\})\geq \delta(\Bar{G})-1$. Let $\Bar{G}-\{v\}= G_{1}\sqcup G_{2} \sqcup \ldots \sqcup G_{r}$ for some integer $r\geq 2$, where $G_{i}$'s are distinct components. Note that, for each $i\in[r]$,
\begin{equation*}
|G_{i}|\geq \delta(\Bar{G}-\{v\}) \geq \delta(\Bar{G})-1 \geq \frac{n+1}{2}+250.
\end{equation*}
 This implies that $r$ can be at most $3$. In both the cases of $r=2$ and $r=3$, applying Lemma~\ref{pancyclic lemma 1} and proceeding with similar argument, we get the contradiction. 
 
 So the only case remain is $\mathop{k}(\Bar{G})=0$, i.e., $\Bar{G}$ is disconnected. But in this case, also we can reach the contradiction using similar arguments.
 
\noindent\textbf{Case II :}  Let $\Delta(G)\geq \frac{3(n+1)}{2}-250$. i.e., there exists a vertex $v$ such that $|N_{G}(v)|=\frac{3(n+1)}{2}-250+p$ for some integer $p\geq 0$. Using Corollary~\ref{maximum degree lemma 2}, we have $\Delta(G)\leq 2n+1$, which implies $p\leq \frac{n+1}{2}+249$. Here, we assume $p$ to be $0$ for the simplicity of argument and computation. In case of non-zero $p$, the same calculations will follow. 

We construct the graph $H= G[N_{G}(v)]$. Then $|V(H)|=\frac{3(n+1)}{2}-250$. Corollary~\ref{intersection lemma 2} implies that $|N_{H}(u)|\leq (n-1)$ for all $u\in V(H)$. Hence $\Delta(H)\leq (n-1)$ which implies
\begin{align*}
\delta(\Bar{H}) \geq (|V(\Bar{H})| - 1)- \Delta(\Bar{H}) &\geq\left(\frac{3(n+1)}{2}-251\right)-(n-1)\\
&=\frac{(n+1)}{2}-249\geq  \frac{\frac{3(n+1)}{2}-250}{4}+250 = \frac{|V(\Bar{H})|}{4}+250
\end{align*}
The inequality holds since we assume $n\geq 3493$\footnote{Here the actual requirement is $n\geq 3491$, but to satisfy the last inequality of \eqref{3493}, we need $n$ to be at least $3493$. As a result, we assume ``$n\geq 3493$".}. We further study the graph $\Bar{H}$. Note that, in the previous proof we used Lemma~\ref{Brandt 2} on $\Bar{H}$ whereas in this proof we shall use Lemma~\ref{Brandt} on $\Bar{H}$. Hence we divide the rest of the proof into three more cases while $\mathop{k}(\Bar{H})\geq 2$; $\mathop{k}(\Bar{H})=1$ and $\mathop{k}(\Bar{H})=0$. However, the cases $\mathop{k}(\Bar{H})=1$ and $\mathop{k}(\Bar{H})=0$ dealt with the similar arguments we used for $\Bar{G}$ in Case~I to reach the contradiction. The argument to deal with the case $\mathop{k}(\Bar{H})\geq 2$, differs from this proof of the previous theorem. Here, we need an extra result to conclude that $\Bar{H}$ is non-bipartite.

We let $\mathop{k}(\Bar{H})\geq 2$; this implies $\Bar{H}$ is $2-$connected. If we can show that $\Bar{H}$ is non-bipartite, then using Lemma~\ref{Brandt}, it would imply that $\Bar{H}$ is either a weakly pancyclic graph or contains cycles of each length $l$, where $l\in [4,\mathop{c}(\Bar{H})]\cap \mathbb{Z}$ and $l\neq 5$. Since $\delta(\Bar{H})> \frac{|V(\Bar{H})|}{4}$, Lemma~\ref{fact check} ensures that $\mathop{g}(\Bar{H})\leq 4$. As $\mathop{c}(\Bar{H})\geq \delta(\Bar{H})\geq \frac{(n+1)}{2}-249 > m$ and $m\geq 7$, in both the cases $\Bar{H}$ contains a $C_{m}$. This contradicts that $\Bar{G}$ is $C_{m}$-free. However, in order to apply Lemma~\ref{Brandt} on $\Bar{H}$, it is first necessary to show that $\Bar{H}$ is non-bipartite which completes the proof. So if possible, let $\Bar{H}$ be bipartite with partition $V(\Bar{H})=A\sqcup B$, where $|A|\geq |B|$. Note that, 
\begin{equation*}
 2|A|\geq |V(\Bar{H})|=\frac{3(n+1)}{2}-250= |A|+|B|\geq 2|B|  
\end{equation*}
which implies for this case  
\begin{equation*}
 |A|\geq \frac{3(n+1)}{4}-125\geq |B|\geq \delta(\Bar{H}) \geq \frac{(n+1)}{2}-249.   
\end{equation*}
Since all the vertices of $A$ are isolated in $\Bar{H}$ and $v$ is adjacent to all of $A$ in $G$, $A\sqcup \{v\}$ forms a $\mathbb{K}_{|A|+1}$ in $G$. Hence $|A|\leq n$. Otherwise, $A\sqcup \{v\}$ contains a copy of $\mathbb{K}_{n+2}$, thus in particular, it contains a copy of $\mathbb{K}_{2,n}$ in $G$, which is a contradiction. So let $|A|= (n-k)$ for some integer $k\geq 0$. $|A|= (n-k)\geq \frac{3(n+1)}{4}-125$ implies that $k\leq \frac{n+1}{4}+124$. We also have, $|B|=|\Bar{H}|-|A|= (\frac{n+1}{2}-249+k)$. Here we construct the set 
\begin{equation*}
X=\left\{x\in V(G)| x\notin V(H); x\neq v\right\}.
\end{equation*}
From the construction, it follows that $|X|= (2n+3)-(\frac{3(n+1)}{2}-250+1)= \frac{n+1}{2}+250$, i.e $X$ is non-empty.

\begin{remark}
 If $|N(v)|=\frac{3(n+1)}{2}-250+p$ for some integer $p\geq0$, then $p\leq\frac{(n+1)}{2}+249$, consequently $X$ is non-empty.   
\end{remark}

\noindent Together with a similar proof to the claim in Theorem~\ref{main theorem 1},  here we claim the following and omit the proof.
 
\noindent{\textsl{Claim} :} There exists $x\in X$ such that $\{x,a\}\in E(\Bar{G})$ and $\{x,b\}\in E(\Bar{G})$ for some $a\in A$ and $b\in B$.

\noindent Using the claim above, we choose $x\in X$ and the corresponding $a\in A$ and $b\in B$. Now we consider the graph ${H_{1}}=G[V(H)\sqcup \{x\}]$. Here $H_{1}$ is a graph of order $\frac{3(n+1)}{2}-249$. Recall that $H=G[N_{G}(v)]$. Therefore, using the Corollary~\ref{intersection lemma 2}, we have $|N_{G}(x)\cap N_{G}(v)|\leq n-1$, i.e.,  $|N_{H}(x)|\leq n-1$. Since $|V(H)|=|N_{H}(x)|+|N_{\Bar{H}}(x)|$, it follows that
\begin{equation*}
|N_{\Bar{H}}(x)|=|V(H)|-|N_{H}(x)|\geq \left(\frac{3(n+1)}{2}-250\right)-(n-1)=\frac{n+1}{2}-248.
\end{equation*}
This implies 
\begin{align*}
\delta(\Bar{H_{1}})&=\min\{|N_{\Bar{H}_{1}}(t)|: t\in V(\Bar{H}_{1})\}\\
&=\min\{|N_{\Bar{H}_{1}}(t)|: t\in V(\Bar{H})\textup{ or } t=x\}\\
&=\min\left\{\min\{|N_{\Bar{H}}(t)|: t\in V(\Bar{H})\},\min\{|N_{\Bar{H}}(t)|:t=x\}\right\}\}\\
&=\min\left\{\delta(\Bar{H}), |N_{\Bar{H}}(x)|\right\}\\
&\geq\min\left\{\frac{n+1}{2}-249, \frac{n+1}{2}-248\right\}=\frac{n+1}{2}-249
\geq\frac{|V(\Bar{H_{1}})|}{4}+250.\label{3493}\tag{$\ast$}
\end{align*}
The last inequality \eqref{3493} holds since $n\geq 3493$ \footnote{In the computation of $\delta(\Bar{H}_{1})$,  we have used the (non-trivial) foundational fact that, $\min A\cup B=\min\{\min A,\min B\}$, where $A$ and $B$ are finite subsets of $\mathbb{R}$.}. Here we make the following claim.

\noindent{\textsl{Claim} :} The complementary graph $\Bar{H_{1}}$  is a non-bipartite graph.

\begin{proof}[\tt{Proof of claim} :]
We prove this result by using the contradiction method. Thus, we assume (if possible) the graph $\Bar{H_{1}}$ is a bipartite graph, along with the partition $V(\Bar{H_{1}})=A_{1}\sqcup B_{1}$, where $|A_{1}|\geq |B_{1}|$. Recall that we assumed $\Bar{H}$ is the bipartite graph with partition $A\sqcup B=V(\Bar{H})$, where $|A|\geq|B|$. We show that all the vertices of $A$ lie within the same part of $V(\Bar{H_{1}})$, i.e either $A\subset A_{1}$ or $A\subset B_{1}$, but not both. If not, then let there exist some $a', a'' \in A$ such that $a'\in A_{1}$ and $a''\in B_{1}$. Since each vertex of $A$ is an isolated vertex in $\Bar{H}$, the neighbours of $a'$ and $a''$ in $\Bar{H_{1}}$ are either $x$ or some vertex in $B$. This implies $N_{\Bar{H_{1}}}(a')\cup N_{\Bar{H_{1}}}(a'')\subseteq (B\sqcup \{x\})$. Hence
\begin{align*}
|N_{\Bar{H_{1}}}(a')\cap N_{\Bar{H_{1}}}(a'')|=&|N_{\Bar{H_{1}}}(a')|+|N_{\Bar{H_{1}}}(a'')|-|N_{\Bar{H_{1}}}(a')\cup N_{\Bar{H_{1}}}(a'')|\\
\geq& \delta(\Bar{H}_{1})+\delta(\Bar{H}_{1})-|B\sqcup\{x\}|\\
\geq& 2\left(\frac{n+1}{2}-249\right)-(|B|+1)\\
=&(n-497)-\left(\frac{n+1}{2}-248+k\right)\\
=&\frac{n+1}{2}-250-k.
\end{align*}
Since $k\leq \frac{n+1}{4}+124$ and $n\geq 3493$, $(\frac{n+1}{2}-250-k)>0$, i.e., there exists some vertex $q$ such that $\{q,a'\}$ and $\{q,a''\}$ both are in $E(\Bar{H_{1}})$. But this is impossible as $a'$ and $a''$ are in two different parts of a bipartite graph according to the assumption. Thus, without loss of generality, we can assume that $A\subseteq A_{1}$ and consequently $a\in A_{1}$. Since $\Bar{H}$ is a bipartite graph, we have for each $b'\in B$, $N_{\Bar{H}}(b')\subseteq A\subseteq A_{1}$ and $|N_{\Bar{H}}(b')|\geq\delta(\Bar{H})\geq(\frac{n+1}{2}-249)$. Since we have assumed $\Bar{H}_{1}$ to be a bipartite graph, we have such $b'\notin A_{1}$. Therefore, $B\subseteq B_{1}$ and consequently $b\in B_{1}$. Thus, the bipartite graph $\Bar{H}_{1}$ with $V(\Bar{H}_{1})=A_{1}\sqcup B_{1}$ contains a vertex $x$ such that it has one neighbour $a\in A\subset A_{1}$ and has one neighbour $b\in B\subset B_{1}$. A contradiction arises. This establishes the claim.
\end{proof}

So far we have established that if $\Bar{H}$ is bipartite, then $\Bar{H_{1}}$ is non-bipartite. We now aim to prove that $\mathop{k}(\Bar{H})\geq 2$ implies $\mathop{k}(\Bar{H_{1}})\geq 2$. If not, then let $\mathop{k}(\Bar{H_{1}})<2$. Note that $\mathop{k}(\Bar{H})\geq 2$ implies that $\Bar{H}$ is connected, and it remains connected upon the deletion of any single vertex. Since $x$ is adjacent to the connected graph $\Bar{H}$ with at least $\frac{n+1}{2}-248$ many vertices in $\Bar{G}$, we conclude $\Bar{H_{1}}$ is a connected graph, which implies $\mathop{k}(\Bar{H_{1}})>0$. If $\mathop{k}(\Bar{H_{1}})=1$, then there exists a vertex $u$ such that $(\Bar{H_{1}}-\{u\})$ is disconnected. Since $\Bar{H}$ is connected, $u\neq x$. If $u\in \Bar{H}$, then it would imply that $(\Bar{H}-\{u\})$ is disconnected, which is not possible as $\mathop{k}(\Bar{H})\geq 2$. Thus, to complete the proof, we again apply Lemma~\ref{Brandt} on $\Bar{H_{1}}$ to show that it is either weakly pancyclic or contains cycles of all lengths between $7$ and $\mathop{c}(\Bar{H_{1}})$. Applying $\epsilon=\frac{1}{4}$ in Lemma~\ref{fact check} implies that $\mathop{g}(\Bar{H_{1}})\leq 4$. Since $\mathop{c}(\Bar{H_{1}})\geq \delta(\Bar{H_{1}})\geq \frac{n+1}{2}-249 > m$ and $m \geq 7$, in both the cases $\Bar{H_{1}}$ contains a $C_{m}$. In particular, $\Bar{G}$ contains a $C_{m}$. Thus, $\Bar{H}$ is non-bipartite. 

Finally, we conclude that $R(\mathbb{K}_{2,n},C_{m})=2n+3$.

\section{Remarks and Conclusion}
It is pertinent to consider the following question: What distinguishes the two proofs? The answer lies in the final claim established in the proof of Theorem~\ref{main theorem 2}. The next question arises: why is this claim not applicable in Theorem~\ref{main theorem 1}? The answer is rooted in the additional edge present in $B_{n}$, which is absent in $\mathbb{K}_{2,n}$. In the case of $B_{n}$, we are unable to apply Lemma~\ref{intersection lemma} to the vertex $x$ to deduce any upper bound on $N_{H}(x)$ as $v$ and $x$ are not adjacent in $G$. In contrast, for $\mathbb{K}_{2,n}$, we have the advantage of this non-adjacency, which allows us to derive information about $\delta(\Bar{H_{1}})$.

\begin{acknowledgement}
The author would like to express gratitude to Niranjan Balachandran and Sanchita Paul for their insightful discussions on earlier versions of this problem. The author also extends thanks to Kaushik Majumder for his substantial assistance in preparing the draft and numerous discussions on this problem. The author states that there are no conflicts of interest.
  
\end{acknowledgement}

\end{document}